\RequirePackage{fix-cm}
\documentclass[11pt]{article}
\usepackage[a4paper,margin=1in]{geometry}
\usepackage{amsmath,amssymb,amsthm,mathtools}
\usepackage{enumitem,indentfirst,microtype}
\usepackage[hidelinks]{hyperref}
\hypersetup{
  pdftitle={On the Tur\'an number for the Cartesian product of an odd cycle and a tree},
  pdfauthor={Ya-Hong Chen, Yong Lu, Quan Sun, Qi Wu, and Long-Tu Yuan}
}
\numberwithin{equation}{section}
\newtheorem{theorem}{Theorem}[section]
\newtheorem{lemma}[theorem]{Lemma}
\newtheorem{corollary}[theorem]{Corollary}
\newcommand{\ex}{\operatorname{ex}}
\newcommand{\EX}{\operatorname{EX}}
\newcommand{\M}{\mathcal M}
\newcommand{\ind}[1]{\overline K_{#1}}
\newcommand{\ess}{\operatorname{ess}}
\title{\fontsize{18}{22}\selectfont On the Tur\'an number for the Cartesian product of\\an odd cycle and a tree}
\author{Ya-Hong Chen\thanks{Department of Mathematics, Lishui University, No.1 Xueyuan Road, Lishui, 323000, Zhejiang, PR China. E-mail: \texttt{yhchen0611@163.com}.}
\qquad Yong Lu\thanks{School of Mathematics and Statistics, Jiangsu Normal University, Xuzhou, Jiangsu 221116, P.R. China. Email: \texttt{luyong@jsnu.edu.cn}. Supported by the National Natural Science Foundation of China (Nos.~12371348 and 12201258).}
\qquad Quan Sun\thanks{School of Mathematical Sciences, East China Normal University, 500 Dongchuan Road, Shanghai 200240, P.R. China. Email: \texttt{52280155016@stu.ecnu.edu.cn}.}
\qquad Qi Wu\thanks{School of Mathematics and Statistics, Jiangsu Normal University, Xuzhou, Jiangsu 221116, P.R. China. Email: \texttt{wuqimath@163.com}.}
\qquad Long-Tu Yuan\thanks{School of Mathematical Sciences, Key Laboratory of MEA (Ministry of Education), Shanghai Key Laboratory of PMMP, and Nantong Institute for Applied Mathematics and Artificial Intelligence, East China Normal University, Shanghai 200241, China. E-mail: \texttt{ltyuan@math.ecnu.edu.cn}.}}
\date{}
\begin{document}
\maketitle
\begin{abstract}
Let \(T\) be a fixed tree on at least two vertices. We determine the
Tur\'an number and all extremal graphs for the Cartesian product of an
odd cycle \(C_{2m+1}\) and \(T\), for every fixed integer \(m\ge1\) and all
sufficiently large \(n\), in terms of its decomposition family. Every
extremal graph is the join of an independent set and an extremal graph
for that family. The proof combines the
Erd\H{o}s--S\'os theorem with weak stability and progressive induction.
\end{abstract}

\section{Introduction}

For a graph \(H\), let \(\ex(n,H)\) be the largest number of edges in a
graph on \(n\) vertices that does not contain \(H\) as a subgraph, and let
\(\EX(n,H)\) be the family of extremal graphs. For a family
\(\mathcal H\) of graphs, \(\ex(n,\mathcal H)\) and
\(\EX(n,\mathcal H)\) are defined by excluding every member of
\(\mathcal H\) as a subgraph. The Erd\H{o}s--Stone--Simonovits theorem
determines the leading term of the Tur\'an number from the chromatic
number~\cite{ErdosSimonovits,ErdosStone}.

For graphs \(H_1,H_2\), the Cartesian product \(H_1\square H_2\) has
vertex set \(V(H_1)\times V(H_2)\). Vertices \((x_1,x_2)\) and
\((y_1,y_2)\) are adjacent if either \(x_1=y_1\) and
\(x_2y_2\in E(H_2)\), or \(x_2=y_2\) and \(x_1y_1\in E(H_1)\).

We briefly review work on Tur\'an numbers for Cartesian products.
Brada\v{c} et al.~\cite{Grid} proved that the Tur\'an number for the
Cartesian product of any fixed path and any fixed tree, each with at
least one edge, is of order \(n^{3/2}\). Gao et al.~\cite{EvenPrism}
obtained the same order \(n^{3/2}\) for the Cartesian product of a fixed
even cycle of length at least eight and an edge. Wang and
Yang~\cite{TreeProducts} proved that the Tur\'an number for the
Cartesian product of any two fixed trees, each with at least one edge,
is of order \(n^{3/2}\).
For the Cartesian product of a fixed star with at least two edges and
an edge, Zhao et al.~\cite{StarEdge} obtained more precise upper and
lower bounds of order \(n^{3/2}\). These results concern Cartesian
products of bipartite graphs. He et al.~\cite{HeLiFeng} determined
the exact Tur\'an number and all extremal graphs for the Cartesian
product of any fixed odd cycle and an edge when \(n\) is sufficiently
large. In this paper, we replace the edge by an arbitrary fixed tree.

An important input is the Erd\H{o}s--S\'os theorem on embedding trees
under a sharp average-degree condition. The conjecture was proposed
by Erd\H{o}s and S\'os in 1963. Ajtai, Koml\'os, Simonovits and
Szemer\'edi later announced a proof for all sufficiently large trees,
but full details were not made publicly available; see
Wood~\cite{WoodES} for this history. In 2026, OpenAI's GPT-6 Astra
discovered a proof of the full conjecture~\cite{GPT6Astra} in the
FrontierMath Erd\H{o}s project of Adamczewski and
Bloom~\cite{AdamczewskiBloom}; a Lean proof was also produced in
that project.
Wood~\cite{WoodES} gave an exposition, and Riordan and
Scott~\cite{RiordanScott} presented a simplified proof and
determined the equality cases.

The theorem already has applications to multicolour Ramsey numbers
of trees; see Wood~\cite[Theorem~2]{WoodES}. In related Tur\'an
problems, Zhu et al.~\cite{ZhuWangZhangZhang} obtained sharp bounds
and extremal graphs for suspensions of balanced trees, assuming the
Erd\H{o}s--S\'os conjecture for all subtrees of the given tree.
This assumption is now guaranteed by the full theorem. The present
paper gives an application to the exact Tur\'an problem for
Cartesian products of an odd cycle and an arbitrary fixed tree.

Our proof uses Simonovits's weak stability lemma~\cite{Simonovits1974}
to describe the approximate structure of extremal graphs. His
progressive induction~\cite{Simonovits1968} then gives the exact
Tur\'an number from a bound on the number of edges lost when a fixed
number of vertices is deleted. Yuan~\cite{YuanWheels} also used
progressive induction to determine the Tur\'an numbers for odd wheels.

Throughout the paper, \(m\ge1\) is a fixed integer, \(C_{2m+1}\) is the cycle
of length \(2m+1\), \(T\) is a fixed tree on \(k\ge2\) vertices, and
\(F=C_{2m+1}\square T\). The decomposition family \(\M(F)\) is defined
in Section~2. For graphs \(J,J'\) with disjoint vertex sets, their \emph{join}
\(J\vee J'\) is obtained by adding every edge between their vertex
sets. We write \(\ind{b}\) for the independent graph on \(b\) vertices
and use \(K_0\) for the empty graph. For every nonnegative integer
\(a\), define \(g(a)=\ex(a,\M(F))\), with \(g(0)=0\) and
\(\EX(0,\M(F))=\{K_0\}\). Let
\(\Psi(n)=\max_{a+b=n}\{ab+g(a)\}\), where the maximum is over
nonnegative integers. Let \(\mathcal G(n)\) be the family of all graphs
\(J\vee\ind{b}\) such that \(a+b=n\),
\(J\in\EX(a,\M(F))\), and \(ab+g(a)=\Psi(n)\).

\begin{theorem}\label{thm:main}
For all sufficiently large \(n\), \(\ex(n,F)=\Psi(n)\) and \(\EX(n,F)=\mathcal G(n)\).
\end{theorem}

We first establish the properties of \(\M(F)\) and the embedding
lemmas. We then show that every extremal graph belongs to
\(\mathcal G(n)\) or admits a deletion of fixed size satisfying a
suitable edge bound. Progressive induction completes the proof.

\section{Preliminaries and notation}

All graphs are finite, simple, and undirected. For a graph \(G\),
write \(v(G)=|V(G)|\) and \(e(G)=|E(G)|\). Its neighborhood and degree
at \(x\) are denoted by \(N_G(x)\) and \(d_G(x)\); the subscript
is omitted when the ambient graph is clear. We use the convention
\(\bigcap_{x\in\varnothing}N_G(x)=V(G)\).
For \(A\subseteq V(G)\),
\(G[A]\) denotes the induced subgraph; we use
\(G-A=G[V(G)\setminus A]\) and \(d_A(x)=|N_G(x)\cap A|\).
If \(x\notin A\), let \(\overline d_A(x)=|A|-d_A(x)\).
We identify an edge with the set of its two endpoints.
For disjoint \(A,B\subseteq V(G)\), denote by \(E(A,B)\) the set of edges between
them, and define \(e(A,B)=|E(A,B)|\) and
\(\overline e(A,B)=|A||B|-e(A,B)\).
The graph \(G[A,B]\) has vertex set \(A\cup B\) and edge set
\(E(A,B)\). A set \(A\) is \emph{complete to} \(B\) if every
possible edge between them is present. We write \(\Delta(G)\)
for the maximum degree, \(\alpha(G)\) for the independence number,
\(\nu(G)\) for the matching number, and
\(\tau(G)\) for the minimum size of a vertex cover.
For a positive integer \(p\), let \([p]=\{1,\ldots,p\}\).

We use the standard hierarchy notation
\(a\gg b>0\): after \(a\) and all other preceding parameters have
been fixed, \(b\) is chosen sufficiently small in terms of them.
Thus a hierarchy such as
\(1\gg\varepsilon_1\gg\varepsilon_2\gg n_0^{-1}>0\)
specifies the order of the choices; it does not assert a numerical
inequality with a universal ratio. All constants may depend on the
fixed graph \(F\), and \(n_0\) is chosen last. Estimates below are
understood for \(n\ge n_0\).

Identify \(V(C_{2m+1})\) with \(\mathbb Z_{2m+1}\), with addition
modulo \(2m+1\). For \(x\in V(T)\), the set
\(L_x=\mathbb Z_{2m+1}\times\{x\}\) is its \emph{layer} in \(F\).
Let \(V(T)=T^0\dot\cup T^1\) be the bipartition of \(T\).
Assign color \(0\) to
\((\{0\}\times T^0)\cup(\{2m\}\times T^1)\).
The remaining vertices of each layer induce a path on \(2m\) vertices.
Color these paths alternately with \(1,2\), starting with color \(1\)
at \((1,x)\) for \(x\in T^0\) and at \((0,x)\) for \(x\in T^1\).
For every \(xy\in E(T)\) and \(a\in\mathbb Z_{2m+1}\), the vertices
\((a,x),(a,y)\) have different colors. Thus the coloring is proper,
with color classes of sizes \(k,mk,mk\).
Since \(F\) contains an odd cycle, \(\chi(F)=3\).

The notion of a decomposition family was introduced by
Simonovits~\cite{Simonovits1983}.

For a graph of chromatic number three \(H\), a graph \(M\) is \emph{feasible}
for \(H\) if \(H\subseteq(M\cup\ind{s})\vee\ind{s}\) for some
positive integer \(s\), with all three vertex sets disjoint. An embedding
uses at most \(v(H)\) vertices in each independent set, so feasibility
is equivalent to \(H\subseteq(M\cup\ind{v(H)})\vee\ind{v(H)}\).
The \emph{decomposition family} \(\M(H)\) consists, up to isomorphism,
of the feasible graphs minimal under ordinary, not necessarily induced,
subgraph containment.
Denote by \(\ess(J)\) the graph obtained from \(J\) by deleting its isolated
vertices. Equivalently, \(M\) is feasible if and only if
\(\ess(H-I)\subseteq M\) for some independent set
\(I\subseteq V(H)\). Thus \(\M(H)\) is the finite family of minimal
graphs among these \(\ess(H-I)\). Every member of
\(\M(H)\) has at most \(v(H)\) vertices.

An \emph{edge blowup} of a tree replaces each vertex \(x\) by an
edge \(e_x\), with these \emph{replacement edges} pairwise disjoint in their vertices,
and adds one \emph{connector edge} for each tree edge \(xy\), joining
arbitrarily chosen endpoints of \(e_x,e_y\). An edge blowup of the tree \(T\) is a tree on \(2k\) vertices with \(2k-1\) edges,
and its replacement edges form a perfect matching. For a forest,
apply this definition to each component.

We use the Erd\H{o}s--S\'os theorem, proved in full by
OpenAI's GPT-6 Astra~\cite{GPT6Astra}, in the formulation of
Riordan and Scott~\cite{RiordanScott}: for every tree \(D\) on at
least two vertices and every nonnegative integer \(a\),
\begin{equation}\label{eq:erdos-sos}
 \ex(a,D)\le\frac{v(D)-2}{2}a.
\end{equation}
We apply this bound only to edge blowups of trees, in
Lemmas~\ref{lem:decomp} and~\ref{lem:pinnededge}.
For paths, the same bound is the classical Erd\H{o}s--Gallai
theorem~\cite{ErdosGallai}.
We also use K\H{o}nig's theorem: the
vertex cover and matching numbers of a bipartite graph are equal.
Hall's theorem states that a bipartite graph with parts \(A,B\) has a
matching covering \(A\) if and only if
\(\bigl|\bigcup_{x\in S}N(x)\bigr|\ge|S|\) for every \(S\subseteq A\).

\begin{lemma}\label{lem:decomp}
Every edge blowup of \(T\) belongs to \(\M(F)\). Every member of
\(\M(F)\) has a connected component
with matching number at least \(k\). For every integer \(a\ge0\),
\((k-1)a-\binom{k}{2}\le g(a)\le(k-1)a\), and
\(\ex(n,F)\ge\Psi(n)\).
For nonnegative integers \(n,t\) with \(n\ge2t\),
\(\Psi(n)-\Psi(n-2t)\ge t(n-t)+(k-1)t-\binom{k}{2}\).
\end{lemma}
\begin{proof}
Let \(D\) be an edge blowup of \(T\). For a cycle edge \(f=\{a,a+1\}\), let
\(I_f=\{a+2,a+4,\ldots,a+2m\}\) and
\(A_f=\mathbb Z_{2m+1}\setminus I_f\).
Then \(I_f\) is independent and \(A_f\) induces only the edge \(f\).
For distinct cycle edges \(f,g\) meeting at \(c\), we have
\(I_f\cap I_g=\varnothing\) and \(A_f\cap A_g=\{c\}\).
Root \(T\) at \(r\). Choose \(f_r\in E(C_{2m+1})\) and a bijection
from the endpoints of \(e_r\) to the coordinates in \(f_r\).
Suppose that \(y\) is a child of \(x\), and that \(ab\) is the connector
edge with \(a\in e_x\), \(b\in e_y\). If \(a\) has coordinate \(c\),
choose for \(f_y\) the cycle edge other than \(f_x\) incident with \(c\).
Assign coordinate \(c\) to \(b\), and the other coordinate of \(f_y\)
to the other endpoint of \(e_y\). Each nonroot vertex has a unique
parent, so each replacement edge receives exactly one assignment.
Denote by \(f_x\) the cycle edge chosen at \(x\), and set
\(I=\bigcup_{x\in V(T)}(I_{f_x}\times\{x\})\).
Then \(I\) is independent, \(|I|=mk\), and \(F-I\) is isomorphic to
\(D\cup\ind{k(m-1)}\). Therefore
\(F\subseteq(D\cup\ind{k(m-1)})\vee\ind{mk}\), and every blowup is
feasible.

For any independent \(I\subseteq V(F)\), at most \(m\) vertices
are deleted from each layer, leaving at least one cycle edge.
For \(xy\in E(T)\), independence also makes the deleted coordinate
sets in \(L_x,L_y\) disjoint. Their complements therefore intersect
in at least one coordinate, leaving at least one edge between those
layers. Thus \(\ess(F-I)\) has at least \(k+(k-1)=2k-1\) edges.
A feasible blowup has exactly this many edges and no isolated
vertices, so none of its proper subgraphs is feasible. It belongs
to \(\M(F)\).

For the connected matching assertion, let \(Q\) be a component of
\(F-I\), and let \(h_x(Q)\) count its cycle edges in \(L_x\).
For \(xy\in E(T)\), group the vertices of
\(F[L_x\cup L_y]-I\) by the pairs
\(P_a=\{(a,x),(a,y)\}\), and write \(r_a=|I\cap P_a|\).
Each \(P_a\setminus I\) is nonempty and connected.
Independence gives \(r_a\in\{0,1\}\), and exactly
\(2-r_a-r_{a+1}\) cycle edges join
\(P_a\setminus I\) to \(P_{a+1}\setminus I\).
A component of \(F[L_x\cup L_y]-I\) consists of the nonempty sets
\(P_a\setminus I\) with consecutive indices around the cycle; there is no edge
between \(P_a\setminus I\) and \(P_{a+1}\setminus I\) exactly when
\(r_a=r_{a+1}=1\). If some such pair has no edge between its sets,
list the indices of each component consecutively from \(a\) to \(b\)
around the cycle. Then \(r_a=r_b=1\). Summing
\(2-r_a-r_{a+1}\) within a component gives an even number: every
interior value is counted twice and the two end values sum to two.
A component consisting of one set \(P_a\setminus I\) has no cycle
edges. If every consecutive pair of these sets has an edge between
them, the total number of cycle edges is
\(2(2m+1)-2\sum_{a\in\mathbb Z_{2m+1}}r_a\), also even.
Since \(Q\cap(L_x\cup L_y)\) is a union of components of
\(F[L_x\cup L_y]-I\), the sum \(h_x(Q)+h_y(Q)\) is even, and hence
\(h_x(Q)\equiv h_y(Q)\pmod2\).

Fix \(x\in V(T)\). Summing \(h_x(Q)\) over the components \(Q\) of
\(F-I\) gives \(2m+1-2|I\cap L_x|\), which is odd. Thus some
component \(Q\) has odd \(h_x(Q)\). Applying the parity equality along
the paths of \(T\), we obtain odd \(h_y(Q)\) for every \(y\in V(T)\).
Choosing one edge of \(Q\) in each layer gives a matching of size
\(k\). The component \(Q\) has edges and remains a component of
\(\ess(F-I)\). If \(M\) is a minimal feasible graph, then
\(\ess(F-I)\subseteq M\) for some independent set \(I\). Since
\(\ess(F-I)\) is itself feasible, the minimality of \(M\) implies that
\(M\) is isomorphic to \(\ess(F-I)\). This proves the assertion for \(\M(F)\).

Write \(a=q(2k-1)+s\), where \(q,s\) are nonnegative integers and
\(0\le s<2k-1\). The disjoint union \(qK_{2k-1}\cup K_s\) does not contain any member of \(\M(F)\) as a subgraph,
since none of its components can contain a connected graph with a
matching of size \(k\). It has
\((k-1)a-s(2k-1-s)/2\ge (k-1)a-\binom{k}{2}\) edges.
By the Erd\H{o}s--S\'os bound applied to a blowup on \(2k\) vertices,
\(g(a)\le(k-1)a\).
If \(J\) does not contain any member of \(\M(F)\) as a subgraph, then
\(J\vee\ind{b}\) does not contain \(F\) as a subgraph:
a copy would yield an independent \(I\subseteq V(F)\) with
\(\ess(F-I)\subseteq J\), containing a member of \(\M(F)\).
For \(a+b=n\) and \(J\in\EX(a,\M(F))\), the graph
\(J\vee\ind{b}\) has \(ab+g(a)\) edges. Maximizing over \(a,b\)
gives \(\ex(n,F)\ge\Psi(n)\).
For the increment bound, choose \((a,b)\) attaining
\(\Psi(n-2t)\). From the linear bounds,
\(g(a+t)-g(a)\ge (k-1)t-\binom{k}{2}\). The pair
\((a+t,b+t)\) satisfies \((a+t)+(b+t)=n\), and
\((a+t)(b+t)-ab=t(a+b)+t^2=t(n-t)\). This proves the last assertion.
\end{proof}

Let \(T_p(n)\)
denote the complete multipartite graph on \(n\) vertices with \(p\)
parts whose sizes differ by at most one.

\begin{theorem}[{Simonovits~\cite[Lemma~3.1.1, pp.~359--360]{Simonovits1974}}]\label{thm:weak}
Let \(\mathcal H=\{H_1,\ldots,H_q\}\), where \(q\ge1\), be a fixed
family of graphs. Set
\(p+1=\min_{1\le i\le q}\chi(H_i)\ge3\) and
\(r=\max_{1\le i\le q}v(H_i)\).
Suppose that, for some \(i\in[q]\) with \(\chi(H_i)=p+1\),
\(H_i\subseteq D\vee T_{p-1}((p-1)r)\) for a fixed tree \(D\). Fix a
nonnegative integer \(\lambda\) and a real number \(\varepsilon>0\).
Then, for all sufficiently large \(n\), the following holds. Let \(G\)
be a graph on \(n\) vertices that does not contain any member of
\(\mathcal H\) as a subgraph and satisfies
\(e(G)\ge\ex(n,\mathcal H)-\lambda n\). Choose
\(V(G)=V_1\dot\cup\cdots\dot\cup V_p\) minimizing
\(\sum_{i\in[p]}e(G[V_i])\). Then:
\begin{enumerate}[label=\textup{(\roman*)},leftmargin=2em]
\item \(\sum_{1\le i<j\le p}\overline e(V_i,V_j)=O(n)\);
\item \(e(G[V_i])=O(n)\) for every \(i\in[p]\);
\item \(\bigl||V_i|-n/p\bigr|=O(\sqrt n)\) for every \(i\in[p]\);
\item there is a constant \(c_\varepsilon\) such that, for every
\(i\in[p]\), at most \(c_\varepsilon\) vertices \(x\in V_i\) have
\(d_{V_i}(x)>\varepsilon n\) or
\(\sum_{j\in[p]\setminus\{i\}}\overline d_{V_j}(x)>\varepsilon n\).
\end{enumerate}
The lower bound on \(n\) and all implicit constants may depend on
\(\mathcal H,D,\lambda,\varepsilon\).
\end{theorem}

\section{Proof of \texorpdfstring{Theorem~\ref{thm:main}}{Theorem 1.1}}

\subsection{Embedding layer edges}

\begin{lemma}\label{lem:coordinates}
Let \(G\) be a graph and let \(\lambda\ge0\). Suppose that \(V(T)=S_0\dot\cup S_1\), that
\(U_0,U_1\subseteq V(G)\) are disjoint, and that there is a family
\((e_x)_{x\in V(T)}\) of edges of \(G\) with pairwise disjoint endpoint sets.
Assume also that the following conditions hold.
\begin{enumerate}[label=\textup{(\alph*)},leftmargin=2.7em,
                  itemsep=0pt,topsep=2pt]
\item for each \(i\in\{0,1\}\) and each \(xy\in E(T[S_i])\),
      \(E(e_x,e_y)\ne\varnothing\);
\item if \(xy\in E(T)\), \(x\in S_0\), and \(y\in S_1\), then
      \(G[e_x,e_y]\) contains a matching of size two;
\item for every \(i\in\{0,1\}\) and every \(x\in S_i\),
      \(e_x\subseteq U_i\);
\item for every \(i\in\{0,1\}\), \(|U_i|>v(F)(\lambda+1)\), and every vertex of
      \(U_i\) has at most \(\lambda\) nonneighbors in \(U_{1-i}\).
\end{enumerate}
Then there is an embedding of \(F\) into \(G\) in which, for every
\(x\in V(T)\), one cycle edge in the layer \(L_x\) is mapped onto
\(e_x\).
\end{lemma}
\begin{proof}
For \(i=0,1\) and each \(xy\in E(T[S_i])\), choose an edge in
\(E(e_x,e_y)\). For every \(xy\in E(T)\) with \(x\in S_0\), \(y\in S_1\), choose
a matching of size two in \(G[e_x,e_y]\). Root \(T\) at a vertex
\(r\). Choose a cycle edge \(f_r\), and define \(\varphi_0\) initially
as a bijection from \(f_r\times\{r\}\) onto \(e_r\).

Process the vertices from the root toward the leaves. Suppose that
\(x\) is the parent of \(y\), and that a cycle edge \(f_x\) has already been mapped
onto \(e_x\). If \(x,y\in S_i\) for some \(i\in\{0,1\}\), let the chosen edge in
\(E(e_x,e_y)\) meet \(e_x\) at the image of a coordinate \(c\in f_x\).
Take for \(f_y\) the other cycle edge incident with \(c\), map
\((c,y)\) to the chosen endpoint of \(e_y\), and map the other endpoint
of \(f_y\times\{y\}\) to the other endpoint of \(e_y\). If \(x\in S_i\) and \(y\in S_{1-i}\) for some \(i\in\{0,1\}\), set \(f_y=f_x\). For each
\(a\in f_x\), map \((a,y)\) to the endpoint in \(e_y\) matched to
\(\varphi_0(a,x)\). Since every nonroot vertex has a
unique parent, this defines an injective map \(\varphi_0\) on
\(D=\bigcup_{x\in V(T)}(f_x\times\{x\})\).
For each \(x\in V(T)\), the edge \(f_x\times\{x\}\) maps to \(e_x\).
If \(xy\in E(T[S_i])\), the only edge of \(F[D]\) between \(L_x,L_y\)
is \(\{(c,x),(c,y)\}\), where \(\{c\}=f_x\cap f_y\); it maps to the chosen edge in \(E(e_x,e_y)\).
If \(xy\in E(T)\) joins \(S_0,S_1\), the two edges
\(\{(a,x),(a,y)\}\), \(a\in f_x=f_y\), map to the chosen matching.
Thus \(\varphi_0\) is an embedding of \(F[D]\).

For \(i=0,1\), define
\(P_i=\bigcup_{x\in S_i}(A_{f_x}\times\{x\}) \cup\bigcup_{x\in S_{1-i}}(I_{f_x}\times\{x\})\),
where \(I_f,A_f\) are as in Lemma~\ref{lem:decomp}.
The sets \(P_0,P_1\) specify which vertices will be mapped into
\(U_0,U_1\), respectively. They partition \(V(F)\), and
\(\varphi_0(D\cap P_i)\subseteq U_i\) for each \(i\in\{0,1\}\).
Moreover,
\begin{equation}\label{eq:template-bipartite}
 E(F[P_0])\cup E(F[P_1])\subseteq E(F[D]).
\end{equation}
The set \(I_f\) is independent, while \(C_{2m+1}[A_f]\) has only the
edge \(f\). If \(xy\in E(T[S_i])\), then \(f_x,f_y\) are distinct
cycle edges meeting in one coordinate \(c\), and
\(A_{f_x}\cap A_{f_y}=\{c\}\),
\(I_{f_x}\cap I_{f_y}=\varnothing\). If \(xy\) joins the two parts,
then \(f_x=f_y\), and equal coordinates in the two layers lie in
different sets \(P_0,P_1\), which proves \eqref{eq:template-bipartite}.

We extend the embedding to the vertices outside \(D\). Suppose that \(\varphi:S\to V(G)\) embeds \(F[S]\), extends
\(\varphi_0\), and satisfies \(\varphi(S\cap P_i)\subseteq U_i\) for
each \(i\in\{0,1\}\), where \(D\subseteq S\subsetneq V(F)\).
Choose \(i\in\{0,1\}\) and \(w\in P_i\setminus S\).
By \eqref{eq:template-bipartite},
\(\varphi(N_F(w)\cap S)\subseteq U_{1-i}\). Hence
\(\bigl|U_i\cap\bigcap_{q\in N_F(w)\cap S}N_G(\varphi(q))\bigr|
\ge |U_i|-v(F)\lambda>|S|\), by condition~\textup{(d)}.
There is therefore a vertex in this intersection outside \(\varphi(S)\);
choose it as \(\varphi(w)\). Every edge from \(w\) to \(S\) is now
preserved, and an edge from \(w\) to a vertex outside \(S\cup\{w\}\)
is checked when that vertex is mapped. Repetition gives an embedding
of \(F\) extending \(\varphi_0\).
\end{proof}

\begin{corollary}\label{cor:pinned}
Let \(G\) be a graph, \(\lambda\ge0\) and \(z\in V(T)\). Let \(A,B\subseteq V(G)\) be disjoint,
and let \(uv\in E(G)\) satisfy \(\{u,v\}\cap A=\varnothing\). For every
\(x\in V(T)\setminus\{z\}\), let \(e_x\in E(G[A])\), with these edges
pairwise disjoint in their vertices. Denote by \(C\) the set
\(N_G(u)\cap N_G(v)\cap A\). Suppose that
\begin{enumerate}[label=\textup{(\alph*)},leftmargin=2.7em,
                  itemsep=0pt,topsep=2pt]
\item for every \(xy\in E(T-z)\), \(E(e_x,e_y)\ne\varnothing\);
\item if \(xz\in E(T)\), then \(G[e_x,\{u,v\}]\) contains a matching
      of size two;
\item \(\min\{|B|,|C|\}>v(F)(\lambda+1)\);
\item every vertex of \(A\) has at most \(\lambda\) nonneighbors in
      \(B\), and every vertex of \(B\setminus\{u,v\}\) has at most
      \(\lambda\) nonneighbors in \(A\).
\end{enumerate}
Then there is an embedding of \(F\) into \(G\) in which one cycle edge of
\(L_z\) is mapped onto \(uv\), and one cycle edge of each other layer
\(L_x\) is mapped onto \(e_x\).
\end{corollary}
\begin{proof}
Root \(T\) at \(z\), set \(e_z=uv\), and repeat the construction of
\(D,P_0,P_1\) in the proof of Lemma~\ref{lem:coordinates}, with
\(S_0=V(T)\setminus\{z\}\) and \(S_1=\{z\}\). Conditions
\textup{(a)--(b)} give the required embedding of \(F[D]\), and
\eqref{eq:template-bipartite} still holds. The two vertices of
\(f_z\times\{z\}\) are mapped to \(u,v\), while every other vertex
of \(D\cap P_0\) is mapped into \(A\).

Extend the embedding by sending \(P_0\setminus D\) into \(C\) and
\(P_1\setminus D\) into \(B\). Suppose that \(\varphi\) is the embedding already defined on a set
\(S\supseteq D\). If \(w\in P_0\setminus S\), then
\(\varphi(N_F(w)\cap S)\subseteq(B\setminus\{u,v\})\cup\{u,v\}\).
Every vertex of \(C\) is adjacent to \(u,v\), and each remaining image
has at most \(\lambda\) nonneighbors in \(C\subseteq A\).
Thus at least \(|C|-v(F)\lambda-|S|>0\) vertices of \(C\setminus\varphi(S)\)
are adjacent to all images of \(N_F(w)\cap S\).
If \(w\in P_1\setminus S\), then
\(\varphi(N_F(w)\cap S)\subseteq A\), and at least
\(|B|-v(F)\lambda-|S|>0\) vertices of \(B\setminus\varphi(S)\) are
adjacent to all these images. Choosing such an image for \(w\) at
each step extends the embedding to \(F\).
\end{proof}

\subsection{Graphs containing a core}

For a positive integer \(\ell\), a triple of disjoint vertex sets
\((R,U,V)\) is a \emph{core with parameter \(\ell\)} if
\(|R|=\tau(T)\), \(|U|=|V|=\ell\), \(U,V\) are independent and
complete to each other, and \(R\) is complete to \(U\cup V\).

\begin{lemma}\label{lem:core}
There are positive integers \(t_0,\ell\), depending only on \(F\),
such that the following holds. Let \(G\) be a graph on \(n\) vertices
that does not contain \(F\) as a subgraph, with \(e(G)\ge\Psi(n)\), and suppose that \(G\) contains
a core with parameter \(\ell\). Then either \(G\in\mathcal G(n)\), or
some \(W\subseteq V(G)\) of size \(2t_0\) satisfies
\(e(G)-e(G-W)\le
t_0(n-t_0)+(k-1)t_0-\binom{k}{2}-1\).
\end{lemma}
\begin{proof}
Fix a core \((R,U,V)\) with parameter \(\ell\). We complete the proof in three steps.

\smallskip
\noindent\emph{Step 1. Enlarge \(R\).}\par
Choose positive constants and positive integers \(c,t_0\) with
\[
 1\gg\eta\gg c^{-1}\gg t_0^{-1}>0,
 \qquad s_0=ct_0,\qquad \ell=(2c)^ks_0.
\]
In particular, \(t_0>2v(F)+\binom{k}{2}\) and
\(\eta t_0/2>\binom{k}{2}\).
Whenever a vertex outside \(R\cup U\cup V\) has at least
\(|U|/c\) neighbors in both \(U,V\), add it to \(R\) and
replace both sets by subsets of its neighbors of common size
\(\lfloor|U|/c\rfloor\). After each step, \(U,V\) remain
independent and complete to each other, and \(R\) remains
complete to \(U\cup V\). Each step retains at least a
\(1/(2c)\) fraction
of each set. Even after \(k\) steps their sizes are at least
\(s_0\). If \(|R|\) reached \(k\), the three color classes of \(F\), of sizes
\(k,mk,mk\), could be mapped into \(R,U,V\), respectively.
The required edges are present because these sets are pairwise complete.
The procedure stops with \(\tau(T)\le r=|R|\le k-1\),
\(|U|=|V|=s\ge s_0\), and every vertex outside
\(R\cup U\cup V\) having fewer than \(s/c\) neighbors in
at least one of \(U,V\).

Let \(X\) consist of the vertices \(x\) outside
\(R\cup U\cup V\) with \(d_U(x)<s/c\) and
\(d_V(x)>(1-\eta)s\); define \(Y\) by interchanging \(U,V\), and let
\(Z=V(G)\setminus(R\cup U\cup V\cup X\cup Y)\).
For \(z\in Z\), the stopping condition and \(z\notin X\cup Y\)
give \(d_U(z)+d_V(z)\le(1-\eta+1/c)s\).

We claim that
\(\nu(G[U,X]),\nu(G[V,Y])\le k-r-1\).
Since \(r\ge\tau(T)\), choose a vertex cover
\(S\subseteq V(T)\) of size \(r\).
Suppose there is a matching of size \(k-r\) in \(G[U,X]\),
and assign its edges to the vertices of \(V(T)\setminus S\).
For \(x\in S\), use the edge \(\{q_x,u_x\}\), with distinct
\(q_x\in R\) and distinct unused \(u_x\in U\).
For a tree edge from \(x\in S\) to \(y\notin S\), join
\(q_x\) to the endpoint in \(U\) of the edge assigned to \(y\).
For each \(xy\in E(T[S])\), use either \(\{q_x,u_y\}\) or
\(\{q_y,u_x\}\). All connector edges exist because \(R\) is complete
to \(U\), and \(V(T)\setminus S\) is independent. Denote the
resulting edge blowup of \(T\) by \(D\).
The vertices of this blowup in \(R\cup U\) are complete to \(V\),
and its at most \(k\) vertices in \(X\) have at least
\((1-k\eta)s\ge v(F)\) common neighbors in \(V\). Choose \(v(F)\) of
these common neighbors and \(v(F)\) vertices of \(U\) outside the blowup.
Together with \(D\), the selected vertices contain
\((D\cup\ind{v(F)})\vee\ind{v(F)}\), and hence \(F\), by
Lemma~\ref{lem:decomp} and the definition of feasibility. This is a
contradiction.
Interchanging \((U,X)\) and \((V,Y)\) gives the other bound.

\smallskip
\noindent\emph{Step 2. Delete \(2t_0\) vertices.}\par
We count the edges removed by a suitable set of \(2t_0\) vertices.
By K\H{o}nig's theorem, choose vertex covers \(C_U,C_V\) of
\(G[U,X],G[V,Y]\), respectively, with \(|C_U|,|C_V|\le k\).
Set \(U_c=U\setminus\bigl(C_U\cup\bigcup_{x\in C_U\cap X}N(x)\bigr)\)
and \(V_c=V\setminus\bigl(C_V\cup\bigcup_{y\in C_V\cap Y}N(y)\bigr)\).
Then \(E(U_c,X)=E(V_c,Y)=\varnothing\), and
\(\min\{|U_c|,|V_c|\}\ge s-k-ks/c\ge(1-2k/c)s\), since \(s\ge c\). Choose subsets \(U'\subseteq U_c\), \(V'\subseteq V_c\), each of size
\(t_0\), independently and uniformly,
and let \(W=U'\cup V'\). For \(z\in Z\),
\[
 \mathbb E d_W(z)
 =t_0\Bigl(\frac{d_{U_c}(z)}{|U_c|}
          +\frac{d_{V_c}(z)}{|V_c|}\Bigr)
 \le t_0\frac{1-\eta+1/c}{1-2k/c}
 \le(1-\eta/2)t_0,
\]
 by the hierarchy \(\eta\gg c^{-1}\).
Let \(\rho=(k-1-r)t_0+
\sum_{z\in X\cup Y\cup Z}(t_0-d_W(z))\).
For every \(x\in X\), \(d_W(x)=d_{V'}(x)\le t_0\); for every
\(y\in Y\), \(d_W(y)=d_{U'}(y)\le t_0\). Therefore
\(\mathbb E\rho\ge(k-1-r)t_0+(\eta/2)t_0|Z|\).
Counting the edges in \(E(U,V)\) incident with \(W\),
in \(E(R,W)\), and in \(E(W,X\cup Y\cup Z)\) gives
\[
 e(G)-e(G-W)=t_0(2s-t_0)+2rt_0+
 \sum_{z\in X\cup Y\cup Z}d_W(z)
 =t_0(n-t_0)+(k-1)t_0-\rho,
\]
where \(n=r+2s+|X\cup Y\cup Z|\).
If \(r<k-1\) or \(Z\ne\varnothing\), then
\(\mathbb E\rho>\binom{k}{2}\). Since \(\rho\) is integer,
some choice satisfies \(\rho\ge\binom{k}{2}+1\), proving
the inequality in Lemma~\ref{lem:core}. The remaining case has
\(r=k-1\) and \(Z=\varnothing\). Assume that no set \(W\) of size \(2t_0\)
satisfies the inequality in Lemma~\ref{lem:core}. Then every choice of \(U',V'\)
above gives \(0\le\rho\le\binom{k}{2}\); fix one such choice.

\smallskip
\noindent\emph{Step 3. Determine the equality case.}\par
Because \(U'\) has no edges to \(X\), \(V'\) has no edges
to \(Y\), and \(U,V\) are complete to each other, we have
\(\rho=\overline e(X\cup(U\setminus U'),V')
+\overline e(Y\cup(V\setminus V'),U')\).
Replace \((U,V,X,Y)\) simultaneously by
\((U',V',X\cup(U\setminus U'),Y\cup(V\setminus V'))\).
The sets \(R,U,V,X,Y\) now partition \(V(G)\), with \(|R|=k-1\) and
\(|U|=|V|=t_0\ge2v(F)+\binom{k}{2}+1\).
The sets \(U,V\) are independent and complete to each other,
\(R\) is complete to \(U\cup V\), and
\(E(U,X)=E(V,Y)=\varnothing\).
Thus
\(\mu=\overline e(X,V)+\overline e(Y,U)=\rho\le\binom{k}{2}\).

We prove \(e(G)\le\Psi(n)-\mu\) by estimating the edges inside
\(X,Y\) and the edges from these sets to \(R\). Define
\(X^+=\{x\in X:d_X(x)>0\}\) and
\(Y^+=\{y\in Y:d_Y(y)>0\}\). There is no edge from \(R\) to
\(X^+\). Otherwise, take \(xx'\in E(G[X])\) and
\(qx\in E(G)\) with \(q\in R\). Let \(z\) be a leaf of
\(T\) with neighbor \(p\), and take a blowup \(D\) of \(T-z\).
The replacement edges of \(D\) form a perfect matching,
so both bipartition classes have \(k-1\) vertices.
Embed these classes in \(R,U\), choosing the class containing
an endpoint of \(e_p\) for \(R\) and mapping that endpoint to \(q\).
Every edge of \(D\) is present because \(R\) is complete to \(U\).
The edge \(xx'\) and the edge \(qx\) extend \(D\) to a blowup
of \(T\).
All vertices of this blowup in \(R\cup U\) are complete to \(V\), and
the total number of missing edges from \(x,x'\) to \(V\) is
at most \(\mu\). The vertices of the blowup therefore have at least
\(|V|-\mu\ge v(F)\) common neighbors in \(V\), while at least
\(v(F)\) vertices of \(U\) lie outside the blowup.
The extended blowup, \(v(F)\) of its common neighbors in \(V\), and
\(v(F)\) vertices of \(U\) outside the blowup contain \(F\), by
Lemma~\ref{lem:decomp} and feasibility. This is a contradiction.
Interchanging \((U,X)\) and \((V,Y)\) gives \(e(R,Y^+)=0\).

Similarly, if \(M\in\M(F)\) lay in \(G[X^+]\),
its at most \(v(F)\) vertices would have at least \(|V|-\mu\)
common neighbors in \(V\); \(M\), \(v(F)\) of these common neighbors,
and \(v(F)\) vertices of \(U\) would together contain a copy of \(F\)
by feasibility. Thus neither
\(G[X^+]\) nor \(G[Y^+]\) contains a member of \(\M(F)\) as a subgraph.
By Lemma~\ref{lem:decomp} and the absence of edges from \(R\)
to \(X^+\), we have
\(e(G[X])+e(R,X)\le g(|X^+|)+(k-1)|X\setminus X^+|
\le(k-1)|X|\), with the analogous bound for \(Y\).
Let \(a=|U|+|X|\), \(b=|V|+|Y|\).
There are at most \(ab-\mu\) edges between \(U\cup X\) and \(V\cup Y\),
exactly \((k-1)(|U|+|V|)\) edges from \(R\) to \(U\cup V\),
and at most \(\binom{k-1}{2}\) edges inside \(R\). It follows that
\(e(G)\le ab+(k-1)(a+b)+\binom{k-1}{2}-\mu
\le\Psi(n)-\mu\). For the second inequality, the graph
\((K_{k-1}\vee\ind{a})\vee\ind{b}\) has
\(ab+(k-1)(a+b)+\binom{k-1}{2}\) edges.
The graph \(K_{k-1}\vee\ind{a}\) has matching number at most
\(k-1\), so it does not contain any member of \(\M(F)\) as a subgraph, by
Lemma~\ref{lem:decomp}. Thus \(e((K_{k-1}\vee\ind{a})\vee\ind{b})\le\Psi(n)\).

Since \(e(G)\ge\Psi(n)\), we have \(e(G)=\Psi(n)\), \(\mu=0\),
and equality in all bounds used to obtain
\(e(G)\le\Psi(n)-\mu\). In particular,
\(e(G[R])=\binom{k-1}{2},\quad e(U\cup X,V\cup Y)=ab\),
and equality holds in
\(e(G[X])+e(R,X)\le(k-1)|X|, \quad e(G[Y])+e(R,Y)\le(k-1)|Y|\).
Thus \(R\) is a clique, \(U\cup X\) is complete to
\(V\cup Y\).

Equality in the bounds for \(e(G[X])+e(R,X)\) gives
\(e(G[X^+])=g(|X^+|)=(k-1)|X^+|\), and \(R\) is complete to
\(X\setminus X^+\). Similarly,
\(e(G[Y^+])=g(|Y^+|)=(k-1)|Y^+|\), and \(R\) is complete to
\(Y\setminus Y^+\).
If \(X^+\ne\varnothing\), then \(\nu(G[X^+])\ge k-1\).
Indeed, for \(k\ge3\),
\(e(G[X^+])=(k-1)|X^+|>(k-2)|X^+|\), so the
Erd\H{o}s--Gallai theorem~\cite{ErdosGallai} gives a copy of
\(P_{2k-2}\) in \(G[X^+]\), containing a matching of size \(k-1\).
For \(k=2\), take any edge. The same argument applies to \(Y^+\).
If \(X^+,Y^+\ne\varnothing\), then \(|T^0|,|T^1|\le k-1\),
so choose disjoint edges in \(G[X]\) for the vertices of \(T^0\)
and in \(G[Y]\) for the vertices of \(T^1\).
Apply Lemma~\ref{lem:coordinates} with \(S_0=T^0\), \(S_1=T^1\),
\(U_0=U\cup X\), \(U_1=V\cup Y\), and \(\lambda=0\).
The sets \(S_0,S_1\) are independent, so condition~\textup{(a)} has no
edges to check. Since \(U_0\) is complete to \(U_1\),
\(G[e_x,e_y]\) is complete bipartite whenever \(xy\in E(T)\),
which verifies condition~\textup{(b)}. Conditions~\textup{(c)--(d)}
follow from the locations of the chosen edges and
\(|U_0|,|U_1|\ge t_0>v(F)\). By
Lemma~\ref{lem:coordinates}, \(G\) contains a copy of \(F\), a
contradiction. Interchanging
\((U,X,a)\) and \((V,Y,b)\) if necessary, assume \(Y^+=\varnothing\).

Since \(Y^+=\varnothing\) and equality holds in the preceding bounds,
\(V\cup Y\) is independent and complete to \(R\cup U\cup X\).
Set \(J=G[R\cup U\cup X]\).
If \(M\subseteq J\) for some \(M\in\M(F)\), choose \(v(F)\)
vertices of \(V\) and \(v(F)\) vertices of \(U\setminus V(M)\).
The sets \(V\) and \(U\setminus V(M)\) each have at least
\(v(F)\) vertices, since \(v(M)\le v(F)\) and \(|U|,|V|\ge2v(F)\).
Since \(V\) is complete to
\(V(M)\cup(U\setminus V(M))\), these vertices would contain a copy of
\(F\) by feasibility.
Thus \(J\) does not contain any member of \(\M(F)\) as a subgraph.
Since \(G=J\vee\ind{b}\),
\(e(G)=bv(J)+e(J)\le bv(J)+g(v(J))\le\Psi(n)\).
Equality gives \(e(J)=g(v(J))\) and
\(bv(J)+g(v(J))=\Psi(n)\), so \(G\in\mathcal G(n)\).
\end{proof}

\subsection{Removing a bounded set of vertices}

\begin{lemma}\label{lem:cleanpair}
For every fixed \(c,\gamma,\eta>0\), there are \(\varepsilon>0\) and an integer
\(n_0\) such that the following holds for every \(n\ge n_0\). Let
\(G\) be a graph on \(n\) vertices that does not contain \(F\) as a
subgraph, and let \(A,B\subseteq V(G)\) be disjoint sets satisfying
\begin{enumerate}[label=\textup{(\alph*)},leftmargin=2.7em,
                  itemsep=0pt,topsep=2pt]
\item \(|A|,|B|\ge cn\) and \(\overline e(A,B)\le\gamma n\);
\item \(\Delta(G[A])+\Delta(G[B])\le\varepsilon n\);
\item \(\max_{x\in A}\overline d_B(x),
       \max_{y\in B}\overline d_A(y)\le\varepsilon n\).
\end{enumerate}
Then \(\min\{e(G[A]),e(G[B])\}\le\eta n\), and neither
\(G[A]\) nor \(G[B]\) contains a member of \(\M(F)\) as a subgraph.
\end{lemma}
\begin{proof}
Choose the constants in the order
\[
 \min\{1,c,\gamma^{-1},\eta\}\gg\varepsilon\gg n_0^{-1}>0.
\]
Suppose first that \(G[A]\) contains \(M\in\M(F)\) as a subgraph.
By \textup{(a), (c)},
\(\bigl|B\cap\bigcap_{x\in V(M)}N(x)\bigr|
\ge cn-v(F)\varepsilon n>2v(F)\).
Choose \(B'\) of size \(v(F)\) in this intersection. Again by
\textup{(a), (c)},
\(\bigl|(A\setminus V(M))\cap\bigcap_{y\in B'}N(y)\bigr|
\ge cn-v(F)\varepsilon n-v(F)>v(F)\).
Choose \(A'\) of size \(v(F)\) in the latter intersection.
Then \(G[V(M)\cup A'\cup B']\) contains
\((M\cup\ind{v(F)})\vee\ind{v(F)}\), and hence \(F\), a contradiction. The same argument applies to \(B\), so every induced
subgraph \(J\) of \(G[A]\) or \(G[B]\) satisfies
\begin{equation}\label{eq:induced-linear}
 e(J)\le(k-1)v(J),
\end{equation}
by Lemma~\ref{lem:decomp}.

Suppose, for a contradiction, that \(e(G[A])>\eta n\) and
\(e(G[B])>\eta n\).
Define \(\mathcal B=\{(e,e')\in E(G[A])\times E(G[B]):
\nu(G[e,e'])<2\}\).
By Hall's theorem, for each \((e,e')\in\mathcal B\), one endpoint of
\(e\) is nonadjacent to both endpoints of \(e'\), or one endpoint of
\(e'\) is nonadjacent to both endpoints of \(e\).
For \(x\in A\), \eqref{eq:induced-linear} bounds the number of edges
in \(G[B\setminus N(x)]\) by \((k-1)\overline d_B(x)\); there is an
analogous bound for \(y\in B\). Summing over the edges incident with
\(x\) or \(y\), we obtain
\[
 |\mathcal B|\le(k-1)\Bigl(\sum_{x\in A}d_A(x)\overline d_B(x)
 +\sum_{y\in B}d_B(y)\overline d_A(y)\Bigr)
 \le2(k-1)\gamma\varepsilon n^2.
\]
For each \(x\in T^0\), choose \(e_x\) uniformly from \(E(G[A])\),
and for each \(x\in T^1\), choose \(e_x\) uniformly from \(E(G[B])\),
with all choices independent. For distinct \(x,y\) in the same
part of \(T\), the probability that \(e_x,e_y\) meet is at most
\(2\varepsilon/\eta\). For \(xy\in E(T)\), the probability that
\((e_x,e_y)\in\mathcal B\) is at most
\(2(k-1)\gamma\varepsilon/\eta^2\).
The union bound therefore gives total failure probability at most
\[
 k^2\varepsilon/\eta+
 2k^2(k-1)\gamma\varepsilon/\eta^2<1,
\]
by the hierarchy. Consequently,
we can choose the edges \(e_x\), \(x\in V(T)\), pairwise disjoint,
with \(\nu(G[e_x,e_y])\ge2\) for every \(xy\in E(T)\).

Apply Lemma~\ref{lem:coordinates} with \(S_i=T^i\) for \(i\in\{0,1\}\),
\(U_0=A\), \(U_1=B\), and \(\lambda=\varepsilon n\).
Condition~\textup{(a)} has no edges to check because \(T^0,T^1\) are
independent. Conditions~\textup{(b)--(c)} follow from the chosen edges.
Condition~\textup{(d)} follows from the hypotheses and
\(\min\{|A|,|B|\}\ge cn>v(F)(\varepsilon n+1)\). By
Lemma~\ref{lem:coordinates}, \(G\) contains a copy of \(F\), a
contradiction.
\end{proof}

A vertex \(z\) of \(T\) is a \emph{centroid} if every component of
\(T-z\) has order at most \(\lfloor k/2\rfloor\). Every tree has a
centroid.

\begin{lemma}\label{lem:pinnededge}
Let \(z\) be a centroid of \(T\), and let \(\sigma\) be the largest
order of a component of \(T-z\). For every fixed \(c,\delta>0\),
there are \(\varepsilon>0\) and an integer \(n_0\) such that the following holds
for every \(n\ge n_0\). Let \(G\) be a graph on \(n\) vertices that
does not contain \(F\) as a subgraph, and let \(A,B\) be disjoint
sets satisfying
\begin{enumerate}[label=\textup{(\alph*)},leftmargin=2.7em,
                  itemsep=0pt,topsep=2pt]
\item \(|A|,|B|\ge cn\);
\item \(\max_{x\in A}\overline d_B(x),
       \max_{y\in B}\overline d_A(y)\le\varepsilon n\);
\item \(G[A]\) does not contain any member of \(\M(F)\) as a subgraph;
\item \(\Delta(G[A])\le\varepsilon n\) and
      \(e(G[A])\ge(\sigma-1+\delta)|A|\).
\end{enumerate}
Then there is no edge \(uv\in E(G-A)\) such that
\(\overline d_A(u)\le\varepsilon n\) and
\(\overline d_A(v)\Delta(G[A])\le\varepsilon n\).
\end{lemma}
\begin{proof}
Choose
\[
 \min\{1,c,\delta\}\gg\varepsilon\gg n_0^{-1}>0.
\]
Suppose that \(uv\in E(G-A)\) satisfies
\(\overline d_A(u)\le\varepsilon n\) and
\(\overline d_A(v)\Delta(G[A])\le\varepsilon n\). We will map one cycle edge of \(L_z\) onto \(uv\). Let \(H\) be the graph on
\(A\cap N(v)\) obtained from \(G[A\cap N(v)]\) by deleting all edges
with both endpoints outside \(N(u)\). No induced subgraph of \(G[A]\) contains a member of \(\M(F)\)
as a subgraph. By Lemma~\ref{lem:decomp},
\[
 e(H)\ge e(G[A])-
 \overline d_A(v)\Delta(G[A])-(k-1)\overline d_A(u)
 \ge e(G[A])-k\varepsilon n.
\]
Every edge of \(H\) has both endpoints adjacent to \(v\), and at
least one endpoint adjacent to \(u\).

Let \(S_1,\ldots,S_t\) be the components of \(T-z\). For each
\(j\in[t]\), let \(D_j\) be an edge blowup of \(S_j\). We embed \(D_1,\ldots,D_t\) successively in \(H\).
Let \(H_j\) be obtained from \(H\) by deleting the vertices used for
\(D_1,\ldots,D_{j-1}\). Fewer than \(2(k-1)\) vertices have been used. Since \(\Delta(H)\le\Delta(G[A])\le\varepsilon n\), deleting
those vertices removes fewer than \(2(k-1)\varepsilon n\) further
edges. Together with \(e(H)\ge e(G[A])-k\varepsilon n\), this gives
\[
 e(H_j)\ge e(G[A])-(3k-2)\varepsilon n
 >(\sigma-1)|A|+\delta|A|/2
 >(v(S_j)-1)v(H_j).
\]
By the Erd\H{o}s--S\'os bound, \(H_j\) contains a copy of \(D_j\).

For each \(x\in V(T)\setminus\{z\}\), write its replacement edge as
\(e_x=\{a_x,b_x\}\), choosing \(a_x\in N(u)\); both endpoints lie in
\(N(v)\). The edges \(ua_x,vb_x\) form a matching between \(uv\) and
\(e_x\) whenever \(xz\in E(T)\).
The density assumption also gives
\(\Delta(G[A])\ge2\delta\), and so
\[
 |N(u)\cap N(v)\cap A|
 \ge |A|-\varepsilon n-\frac{\varepsilon n}{2\delta}
 \ge\Bigl(c-\varepsilon-\frac{\varepsilon}{2\delta}\Bigr)n.
\]
Thus \(|N(u)\cap N(v)\cap A|,|B|>v(F)(\varepsilon n+1)\).
The replacement edges in the copies of \(D_1,\ldots,D_t\) satisfy
Corollary~\ref{cor:pinned}\textup{(a)}, and the edges \(ua_x,vb_x\)
verify \textup{(b)}. The common neighborhood bound verifies
\textup{(c)}, while the hypothesis on \(A,B\) gives \textup{(d)}.
Corollary~\ref{cor:pinned}, with \(\lambda=\varepsilon n\), therefore
gives an embedding of \(F\) into \(G\), a contradiction.
\end{proof}

A partition \(V(G)=X\dot\cup Y\) is \emph{admissible with parameter
\(\gamma\)} if it
minimizes
\(e(G[X])+e(G[Y])\) and satisfies
\[
 \bigl||X|-n/2\bigr|\le \gamma\sqrt n,
 \quad
 e(G[X])+e(G[Y])+\overline e(X,Y)\le \gamma n,
\]
together with
\begin{equation}\label{eq:admissible-local}
 \overline d_Y(x)\le d_X(x)+\gamma\quad(x\in X),
 \quad
 \overline d_X(y)\le d_Y(y)+\gamma\quad(y\in Y).
\end{equation}

\begin{lemma}\label{lem:cleaning}
Fix \(\gamma\ge1\), a positive integer \(\ell\), and \(\eta>0\).
There are positive integers \(s,n_0\), depending only on the graph \(F\) and
the constants \(\gamma,\ell,\eta\), such that the following holds for
\(n\ge n_0\). Let \(G\) be a graph on \(n\) vertices
that does not contain \(F\) as a subgraph, with \(e(G)\ge\Psi(n)\),
and let \(X\dot\cup Y\) be a partition admissible with parameter \(\gamma\).
If \(G\) contains no core with parameter \(\ell\), then, after possibly interchanging
\(X,Y\), there are sets \(P\subseteq R\subseteq V(G)\), \(|R|\le s\),
such that, with \(A=X\setminus R\) and \(B=Y\setminus R\), the
following hold.
\begin{enumerate}[label=\textup{(\roman*)},leftmargin=2em]
\item Neither \(G[A]\) nor \(G[B]\) contains a member of \(\M(F)\)
      as a subgraph, \(\Delta(G[A])+\Delta(G[B])\le\eta n\),
      \(e(G[A])\ge(k-\tau(T))n/2-\eta n\), and \(e(G[B])\le\eta n\).
\item \(G[A\cup P]\) does not contain any member of \(\M(F)\) as a
      subgraph, and \(e(B,R)\le |B||P|+\eta n\).
\item Either \(P=R\) and \(|R\cap Y|\Delta(G[A])\le\eta n\), or
      there is an edge \(uv\in E(G)\) with \(u\in R\setminus P\),
      \(v\in B\), and \(d_Y(v)\Delta(G[A])\le\eta n\).
\end{enumerate}
\end{lemma}
\begin{proof}
We may assume \(0<\eta<1\). All constants below depend only on
\(F,\gamma,\ell,\eta\). Choose
\[
 \eta\gg\xi\gg\varepsilon\gg\varepsilon_0>0,
 \qquad s_0=\lceil2\gamma/\varepsilon_0\rceil,
\]
where \(\varepsilon\) is small enough for
Lemma~\ref{lem:cleanpair} with \(c=1/3\) and error \(\xi\).
We fix \(n_0\) only after the finite sequence of constants in
Step~2 has been chosen. All estimates are then uniform over the
possible choices of \(\beta,\rho\) from this sequence.
For \(v\in X\), define \(d_{\mathrm{in}}(v)=d_X(v)\), and for
\(v\in Y\), define \(d_{\mathrm{in}}(v)=d_Y(v)\). Let
\(R_0=\{v\in V(G):d_{\mathrm{in}}(v)>\varepsilon_0n\}\),
\(A_0=X\setminus R_0\), and \(B_0=Y\setminus R_0\).
Since \(\sum_{v\in V(G)}d_{\mathrm{in}}(v)\le2\gamma n\), we have
\(|R_0|\le s_0\). For \(x\in A_0\),
\[
 \overline d_{B_0}(x)\le\varepsilon_0n+\gamma
 \le2\varepsilon_0n<\varepsilon n,
\]
by \eqref{eq:admissible-local}; the same bound holds for
\(\overline d_{A_0}(y)\) when \(y\in B_0\). Moreover,
\(|A_0|,|B_0|\ge n/3\),
\(\Delta(G[A_0])+\Delta(G[B_0])\le2\varepsilon_0n<\varepsilon n\), and
\(\overline e(A_0,B_0)\le\gamma n\).
Thus \(A_0,B_0\) satisfy every hypothesis of
Lemma~\ref{lem:cleanpair}. After interchanging \(X,Y\), if necessary,
\(e(G[B_0])\le\xi n\), and neither \(G[A_0]\) nor \(G[B_0]\)
contains a member of \(\M(F)\) as a subgraph. By
Lemma~\ref{lem:decomp}, \(e(H)\le(k-1)v(H)\) for every induced
subgraph \(H\) of \(G[A_0]\) or \(G[B_0]\). Applying Tur\'an's theorem
to the complement of a nonempty \(H\), we obtain
\(\alpha(H)\ge v(H)^2/(2e(H)+v(H))\ge v(H)/(2k-1)\).
Thus every such subgraph on at least \((2k-1)\ell\) vertices contains an
independent set of order \(\ell\). We complete the proof in four steps.

\smallskip
\noindent\emph{Step 1. Common neighborhoods and edge counts.}\par
Let \(A\subseteq A_0\), \(B\subseteq B_0\), and let
\(K\subseteq V(G)\setminus(A\cup B)\) have size \(\tau(T)\). Define
\(A_K=A\cap\bigcap_{u\in K}N(u), \quad B_K=B\cap\bigcap_{u\in K}N(u)\).
Then
\(|A_K||B_K|\le2\ell(\gamma+2k-1)n\).
If \(|A_K|<2(2k-1)\ell\), then
\(|A_K||B_K|<2(2k-1)\ell n\le2\ell(\gamma+2k-1)n\). Otherwise, at least half the
vertices of \(A_K\) have at most \(2\gamma n/|A_K|\) nonneighbors in
\(B_K\). Choose an independent set \(U\) of order \(\ell\) among them.
At least \(|B_K|-2\ell\gamma n/|A_K|\) vertices of \(B_K\) are complete to
\(U\). If at least \((2k-1)\ell\) such vertices existed, they would
contain an independent set \(V\) of order \(\ell\), and \(K,U,V\) would
form a core with parameter \(\ell\), contrary to the hypothesis. Therefore
\(|B_K|-2\ell\gamma n/|A_K|<(2k-1)\ell\), which gives the bound for
\(|A_K||B_K|\).

We next bound the edges remaining after removing a set of fixed size.
Fix an integer \(s'\ge s_0\). The estimates below hold above a threshold depending on \(s'\);
we will apply them with \(s'=s\) from Step~2. Let
\(R_0\subseteq R'\subseteq V(G)\) with \(|R'|\le s'\), and set
\(A'=X\setminus R'\), \(B'=Y\setminus R'\). Then
\(e(G[A'])+e(G[B'])\ge(k-\tau(T))n/2-\xi n\). At most
\(6\ell(\gamma+2k-1)\binom{s'}{\tau(T)}\)
vertices \(b\in B'\) have at least \(\tau(T)\) neighbors in
\((R'\setminus R_0)\cap Y\).

To bound the number of these vertices \(b\in B'\), fix
\(K\subseteq(R'\setminus R_0)\cap Y\) with \(|K|=\tau(T)\). Each \(u\in K\) has
\(d_{\mathrm{in}}(u)\le\varepsilon_0n\). By
\eqref{eq:admissible-local}, every such \(u\) has at most
\(\varepsilon_0n+\gamma\) nonneighbors in \(X\). Hence
\[
 \left|A'\cap\bigcap_{u\in K}N(u)\right|
 \ge |X|-s'-\tau(T)(\varepsilon_0n+\gamma)>n/3,
\]
by the hierarchy and \(|X|=n/2+O(\sqrt n)\).
With \(A=A'\) and \(B=B'\), the bound for \(|A_K||B_K|\) gives
\(|B_K|\le6\ell(\gamma+2k-1)\). Summing over the at most
\(\binom{s'}{\tau(T)}\) choices of \(K\) proves the stated bound on the number of such vertices.

For the lower bound on \(e(G[A'])+e(G[B'])\), let
\(t(x,y)=|N(x)\cap N(y)\cap R'|\) for
\(x\in A'\), \(y\in B'\), and set
\(c_1=2\ell(\gamma+2k-1)\binom{s'}{\tau(T)}\).
By the bound for \(|A_K||B_K|\), at most \(c_1n\) pairs satisfy
\(t(x,y)\ge\tau(T)\). Since \(t(x,y)\le s'\), counting the
triples \((u,x,y)\) with \(u\in R'\), \(x\in A'\), \(y\in B'\), and
\(ux,uy\in E(G)\), gives
\[
 \sum_{u\in R'}d_{A'}(u)d_{B'}(u)
 =\sum_{(x,y)\in A'\times B'}t(x,y)
 \le(\tau(T)-1)|A'||B'|+s'c_1n.
\]
Because \(|A'|,|B'|\ge n/3\),
\[
 \sum_{u\in R'}
 \frac{d_{A'}(u)d_{B'}(u)}{|A'||B'|}
 \le\tau(T)-1+\frac{9s'c_1}{n}.
\]
For \(u\in R'\cap X\), set
\(w(u)=d_{A'}(u)-\overline d_{B'}(u)\); for
\(u\in R'\cap Y\), set
\(w(u)=d_{B'}(u)-\overline d_{A'}(u)\). For every \(u\in R'\), let
\(a_u=d_{A'}(u)/|A'|\) and \(b_u=d_{B'}(u)/|B'|\).
Since \(0\le a_u,b_u\le1\), we have \(a_u+b_u-1\le a_u b_u\).
For \(u\in R'\cap X\),
\(w(u)\le |A'|a_u b_u+\bigl||A'|-|B'|\bigr|\);
for \(u\in R'\cap Y\),
\(w(u)\le |B'|a_u b_u+\bigl||A'|-|B'|\bigr|\).
Since \(s'\) is fixed and
\(\bigl||A'|-|B'|\bigr|\le2\gamma\sqrt n+s'\), we obtain
\[
 \sum_{u\in R'}w(u)
 \le\frac{\tau(T)-1}{2}n+O(\sqrt n).
\]
The exact edge count is
\[
 e(G)=|X||Y|+e(G[A'])+e(G[B'])-\overline e(A',B')
 +\sum_{u\in R'}w(u)+e(G[R'])-|R'\cap X||R'\cap Y|.
\]
On the other hand, Lemma~\ref{lem:decomp} gives
\(e(G)\ge\Psi(n)\ge |X||Y|+(k-1)|X|-\binom{k}{2}\).
Using \(\overline e(A',B')\ge0\), \(|R'|\le s'\), and the bound for
\(\sum_{u\in R'}w(u)\) in the exact edge count, we obtain
\[
 \begin{aligned}
 e(G[A'])+e(G[B'])
 &\ge(k-1)|X|-\frac{\tau(T)-1}{2}n-O(\sqrt n)\\
 &\ge\frac{k-\tau(T)}{2}n-O(\sqrt n).
 \end{aligned}
\]
The second inequality uses \(|X|=n/2+O(\sqrt n)\).
Since \(n_0\) is chosen last, the error is at most \(\xi n\), proving
\(e(G[A'])+e(G[B'])\ge(k-\tau(T))n/2-\xi n\).

\smallskip
\noindent\emph{Step 2. Choose the degree bound \(\beta n\).}\par
Put \(q_0=s_0+\tau(T)-1\). Choose a positive integer \(t_1\)
and a constant \(\theta>0\) with
\[
 \gamma^{-1}\gg t_1^{-1}>0,
 \qquad \min\{\eta,(q_0+t_1+1)^{-1}\}\gg\theta>0.
\]
We will choose \(\beta,\rho>0\), a positive integer \(s\), and sets
\(R_0\subseteq R^-\subseteq R\subseteq R^+, \quad W=Y\setminus R^+\),
such that \(R=\{v\in V(G):d_{\mathrm{in}}(v)>\beta n\}\), \(|R^+|\le s\),
\(|W|\ge n/3\), and
\begin{equation}\label{eq:scale-summary}
 2^{|R^-|}\rho\le\theta,
 \qquad \rho\gg\beta,
 \qquad \eta\gg(|R^-|+s_0+1)\beta.
\end{equation}
Define
\(E=\bigcup_{u\in(R\setminus R^-)\cap X}(W\setminus N(u)) \cup\bigcup_{u\in(R\setminus R^-)\cap Y}(W\cap N(u))\).
We will also ensure that \(|E|\le\theta|W|\).

To obtain these sets, choose an integer \(p\) with
\(\theta,\gamma^{-1}\gg p^{-1}>0\), and choose the finite hierarchy
\[
 \varepsilon_0\gg\varepsilon_1\gg\cdots\gg\varepsilon_p>0
\]
recursively as follows. For \(0\le i<p\), once \(\varepsilon_i\)
is fixed, put \(s_i=\lceil2\gamma/\varepsilon_i\rceil\) and
\(\rho_i=\theta2^{-s_i}\), and choose \(\varepsilon_{i+1}\) so that
\(\rho_i\gg\varepsilon_{i+1}\) and
\(\eta\gg(s_i+s_0+1)\varepsilon_{i+1}\).
Finally put \(s_p=\lceil2\gamma/\varepsilon_p\rceil\) and choose
\(n_0\) sufficiently large in terms of all these constants.
For \(0\le i\le p\), let
\(R_i=\{v\in V(G):d_{\mathrm{in}}(v)>\varepsilon_i n\}\), so
\(|R_i|\le s_i\), and set \(W=Y\setminus R_p\).
Then \(|X|-s_p,|Y|-s_p\ge n/3\), so \(|W|\ge n/3\).
For each \(i\in\{0,\ldots,p-1\}\), set \(H_i=R_{i+1}\setminus R_i\)
and define
\(E_i=\bigcup_{u\in H_i\cap X}(W\setminus N(u)) \cup\bigcup_{u\in H_i\cap Y}(W\cap N(u))\).
The sets \(H_i\) are pairwise disjoint. If \(u\in H_i\cap X\),
then
\(|W\setminus N(u)|\le\overline d_Y(u)\le d_X(u)+\gamma\); if
\(u\in H_i\cap Y\), then \(|W\cap N(u)|\le d_Y(u)\). Thus
\[
 \sum_{i=0}^{p-1}|E_i|
 \le\sum_{u\in R_p}d_{\mathrm{in}}(u)+\gamma|R_p\cap X|
 \le2\gamma n+\gamma s_p.
\]
Averaging and the choices of \(p,n_0\) give
\[
 \min_{0\le i<p}|E_i|
 \le \frac{2\gamma n+O(1)}p<\theta n/3\le\theta|W|.
\]
Choose
\(j\in\{0,\ldots,p-1\}\) with \(|E_j|\le\theta|W|\), and set
\(R^-=R_j,\quad R=R_{j+1},\quad R^+=R_p, \quad \beta=\varepsilon_{j+1},\quad \rho=\rho_j,\quad s=s_p\).
The hierarchy and \(|R^-|\le s_j\) imply
\eqref{eq:scale-summary}. In particular, \(\beta\ge\varepsilon_p\)
and \(\rho\ge\rho_{p-1}\); these lower bounds depend only on the
fixed parameters.

Set \(A=X\setminus R\) and \(B=Y\setminus R\). Then
\(d_{\mathrm{in}}(v)\le\beta n\) for every \(v\in A\cup B\),
and \(W\subseteq B\).

\smallskip
\noindent\emph{Step 3. Choose \(P\).}\par
Let
\(Z=\{b\in W:d_{(R^+\setminus R_0)\cap Y}(b)\ge\tau(T)\}\).
The bound in Step~1, applied with \(R'=R^+\) and \(s'=s\), gives
\(|Z|\le c_Z:=6\ell(\gamma+2k-1)\binom{s}{\tau(T)}\).
For \(b\in W\), define
\(q(b)=d_R(b)-|R\cap X| =d_{R\cap Y}(b)-\overline d_{R\cap X}(b)\).
If \(b\notin Z\), then \(b\) has at most \(s_0\) neighbors in
\(R_0\cap Y\) and at most \(\tau(T)-1\) neighbors in
\((R^+\setminus R_0)\cap Y\), so \(q(b)\le q_0\). For every \(b\in W\),
\(q(b)\ge-\overline d_{R^+\cap X}(b)\). By \eqref{eq:admissible-local},
\(\sum_{u\in R^+\cap X}\overline d_W(u)
\le\sum_{u\in R^+\cap X}(d_X(u)+\gamma)
\le2\gamma n+\gamma s\). Consequently,
\[
 |\{b\in W:q(b)<-t_1\}|
 \le\frac{2\gamma n+\gamma s}{t_1}.
\]
By the choices of \(t_1,n_0\), the right side is smaller than
\(|W|/4\). Hence at least \(3|W|/4\) vertices satisfy
\(q(b)\ge-t_1\).

For \(S\subseteq R^-\), let
\(W_S=\{b\in W\setminus(E\cup Z):N(b)\cap R^-=S\}\).
Every vertex of \(W\setminus E\) is adjacent to every vertex of
\((R\setminus R^-)\cap X\) and to no vertex of
\((R\setminus R^-)\cap Y\). Therefore \(q(b)=|S|-|R^-\cap X|\) for every \(b\in W_S\).
Define \(\mathcal S=\{S\subseteq R^-:|W_S|\ge\rho|W|\}\).
There are at most \(2^{|R^-|}\) sets \(W_S\); hence
\(\bigl|W\setminus(E\cup Z\cup\bigcup_{S\in\mathcal S}W_S)\bigr|
\le2^{|R^-|}\rho|W|\le\theta|W|\), by \eqref{eq:scale-summary}.
Moreover,
\(\bigl|W\setminus\bigcup_{S\in\mathcal S}W_S\bigr|
\le2\theta|W|+c_Z<3|W|/4\).
Since at least \(3|W|/4\) vertices satisfy \(q(b)\ge-t_1\), some
\(S\in\mathcal S\) satisfies
\(|S|-|R^-\cap X|\ge-t_1\). Choose \(S_*\in\mathcal S\) maximizing
\(|S|-|R^-\cap X|\), and define
\(q_*=|S_*|-|R^-\cap X|, \quad P=S_*\cup((R\setminus R^-)\cap X)\).
Then \(q_*\ge-t_1\), \(|P|=|R\cap X|+q_*\), and \(P\subseteq N(b)\) for every \(b\in W_{S_*}\).

For every \(S\in\mathcal S\) and \(b\in W_S\), we have \(q(b)\le q_*\).
The set \(W\setminus(Z\cup\bigcup_{S\in\mathcal S}W_S)\) has size at
most \(2\theta|W|\); each of its vertices satisfies \(q(b)\le q_0\).
The vertices of \(Z\) satisfy \(q(b)\le |R|\le s\). Using also
\(q_*\ge-t_1\), we obtain
\[
 e(W,R)=|W||R\cap X|+\sum_{b\in W}q(b)
 \le |W||P|+2\theta(q_0+t_1)|W|+(s+t_1)c_Z.
\]
Moreover, \(|B\setminus W|\le s\) and \(|R|\le s\). Thus
\[
 e(B,R)\le |B||P|+2\theta(q_0+t_1)n+O(1)
 \le |B||P|+\eta n,
\]
by the hierarchy and the choice of \(n_0\).

Suppose that \(M_*\in\M(F)\) lies in \(G[A\cup P]\). Every vertex
of \(A\) has at most \(\beta n+\gamma\) nonneighbors in \(B\), while
\(W_{S_*}\) is complete to \(P\). Using \(|W_{S_*}|\ge\rho|W|\), \(|W|\ge n/3\), and
\(\rho\gg\beta\), we have
\[
 \bigl|W_{S_*}\cap\bigcap_{x\in V(M_*)}N(x)\bigr|
 \ge\rho|W|-v(F)(\beta n+\gamma)
 >v(F),
\]
by the hierarchy.
Choose \(B_*\subseteq W_{S_*}\cap\bigcap_{x\in V(M_*)}N(x)\)
of size \(v(F)\). For every \(y\in B_*\), we have \(d_Y(y)\le\beta n\),
so \eqref{eq:admissible-local} gives
\(\overline d_A(y)\le\beta n+\gamma\). Therefore, using \(|A|\ge n/3-s\) and
the hierarchy,
\[
 \bigl|(A\setminus V(M_*))\cap\bigcap_{y\in B_*}N(y)\bigr|
 \ge |A|-v(F)(\beta n+\gamma)-v(F)>v(F),
\]
by the choice of \(n_0\).
Choose \(A_*\subseteq(A\setminus V(M_*))\cap\bigcap_{y\in B_*}N(y)\)
of size \(v(F)\). Then \(G[V(M_*)\cup A_*\cup B_*]\) contains
\((M_*\cup\ind{v(F)})\vee\ind{v(F)}\), so it contains \(F\), a
contradiction. Thus \textup{(ii)} holds.

\smallskip
\noindent\emph{Step 4. Verify the three conclusions.}\par
Since \(A\subseteq A_0\) and \(B\subseteq B_0\), neither induced graph
contains a member of \(\M(F)\) as a subgraph. Apply the lower bound on
\(e(G[A'])+e(G[B'])\) with \(R'=R\) and \(s'=s\).
Together with the edge bound for \(G[B_0]\), we obtain
\(e(G[A])\ge(k-\tau(T))n/2-2\xi n \ge(k-\tau(T))n/2-\eta n, \quad e(G[B])\le\eta n\).
Also,
\(\Delta(G[A])+\Delta(G[B])\le2\beta n<\eta n\), proving
\textup{(i)}.

If \(P=R\), then \((R\setminus R^-)\cap Y=\varnothing\), so
\(|R\cap Y|\Delta(G[A]) \le |R^-|\beta n<\eta n\).
Suppose that \(P\ne R\), and choose \(u\in R\setminus P\). If \(u\in X\), minimality of \(e(G[X])+e(G[Y])\) gives
\(d_Y(u)\ge d_X(u)=d_{\mathrm{in}}(u)\); if \(u\in Y\), then
\(d_Y(u)=d_{\mathrm{in}}(u)\). In both cases,
\(d_W(u)\ge d_Y(u)-|R^+|>\beta n-s>\beta n/4\),
where the last inequality follows from \(\beta\ge\varepsilon_p\)
and the choice of \(n_0\).
Since \(e(G[B_0])\le\xi n\),
\(\bigl|\{b\in B_0:d_{B_0}(b)>8\xi/\beta\}\bigr|<\beta n/4\).
We may therefore choose \(v\in N(u)\cap W\) with
\(d_{B_0}(v)\le8\xi/\beta\). Since
\(Y=B_0\dot\cup(R_0\cap Y)\) and \(|R_0|\le s_0\), it follows that
\(d_Y(v)\le 8\xi/\beta+s_0\).
Since \(\Delta(G[A])\le\beta n\),
\[
 d_Y(v)\Delta(G[A])
 \le(8\xi/\beta+s_0)\beta n
 =(8\xi+s_0\beta)n<\eta n,
\]
by the hierarchy and \eqref{eq:scale-summary}. Thus \(uv\) satisfies the second
alternative in \textup{(iii)}.
\end{proof}

\subsection{Reduction and progressive induction}

\begin{lemma}\label{lem:reduction}
Let \(t_0,\ell\) be the constants in Lemma~\ref{lem:core}. There is an
integer \(n_1=n_1(F)\) such that, for every \(n\ge n_1\) and every
\(G\in\EX(n,F)\), either \(G\in\mathcal G(n)\), or there is a set
\(S\subseteq V(G)\) with \(|S|=2t_0\) satisfying
\(e(G)-e(G-S)\le t_0(n-t_0)+(k-1)t_0-\binom{k}{2}-1\).
\end{lemma}
\begin{proof}
Let \(G\in\EX(n,F)\), where \(n\ge n_1\), with \(n_1\) chosen
below depending only on \(F\). Suppose that
\(G\notin\mathcal G(n)\) and that no set \(S\subseteq V(G)\) of size
\(2t_0\) satisfies the inequality in the statement. Fix a blowup \(D\) of
\(T\). By the construction in the proof of Lemma~\ref{lem:decomp},
we have
\(F\subseteq\bigl(D\cup\ind{k(m-1)}\bigr)\vee\ind{mk}\).
Attach \(k(m-1)\) new leaves to \(D\), obtaining a tree \(D'\). Then
\(D\cup\ind{k(m-1)}\subseteq D'\) as an ordinary, not necessarily
induced, subgraph, and \(F\subseteq D'\vee\ind{v(F)}\). Fix
\(1\gg\varepsilon_0>0\). Apply Theorem~\ref{thm:weak} with
\[
 \mathcal H=\{F\},\quad p=2,\quad r=v(F),\quad
 \lambda=0,\quad \varepsilon=\varepsilon_0,\quad D=D'.
\]
Here \(T_{p-1}((p-1)r)=T_1(v(F))=\ind{v(F)}\), so the hypotheses hold.
By Theorem~\ref{thm:weak}, there is a partition
\(V(G)=X\dot\cup Y\) minimizing \(e(G[X])+e(G[Y])\), with
\(|X|,|Y|=n/2+O(\sqrt n)\) and
\(e(G[X])+e(G[Y])+\overline e(X,Y)=O(n)\). Let
\(X_0=\{x\in X:d_X(x),\overline d_Y(x)\le\varepsilon_0n\}, \quad Y_0=\{y\in Y:d_Y(y),\overline d_X(y)\le\varepsilon_0n\}\).
By Theorem~\ref{thm:weak}\textup{(iv)},
\(|X\setminus X_0|+|Y\setminus Y_0|=O(1)\). For \(x\in X\), define
\(G'\) on \(V(G)\) by
\(E(G')=E(G-x)\cup\{xy:y\in Y_0\}\). Any copy \(F'\) of \(F\) in
\(G'\) must contain \(x\), and \(W=N_{F'}(x)\subseteq Y_0\). But
\[
 \bigl|(X_0\setminus V(F'))\cap\bigcap_{y\in W}N_G(y)\bigr|
 \ge |X_0|-v(F)\varepsilon_0n-v(F)>0.
\]
Replacing \(x\) by a vertex in this set gives a copy of \(F\) in
\(G\). The graph \(G'\) does not contain \(F\) as a subgraph, so extremality gives
\(d_G(x)\ge|Y_0|\), and hence
\(\overline d_Y(x)\le d_X(x)+|Y\setminus Y_0|\). Interchanging
\(X\) and \(Y\) gives a fixed
\(\gamma\ge1\) such that, for every \(x\in X\) and \(y\in Y\),
\begin{equation}\label{eq:replacement}
 \overline d_Y(x)\le d_X(x)+\gamma,
 \quad
 \overline d_X(y)\le d_Y(y)+\gamma.
\end{equation}
Increase \(\gamma\), if necessary, so that
\(\bigl||X|-n/2\bigr|\le \gamma\sqrt n\) and
\(e(G[X])+e(G[Y])+\overline e(X,Y)\le \gamma n\). The partition
\(X\dot\cup Y\) is then admissible with parameter \(\gamma\).

By Lemma~\ref{lem:core}, the assumptions that
\(G\notin\mathcal G(n)\) and that no set of size \(2t_0\) satisfies
the deletion inequality imply that \(G\) contains no core with parameter \(\ell\). Choose a centroid \(z\) of \(T\), and let \(\sigma\) be
the largest order of a component of \(T-z\). Let \(\varepsilon>0\)
be given by Lemma~\ref{lem:pinnededge} with \(c=1/3\) and
\(\delta=1/2\). Choose
\[
 \min\{1,\varepsilon,\gamma^{-1}\}\gg\eta>0,
\]
and let \(s\) be supplied by Lemma~\ref{lem:cleaning} for this
\(\eta\). Finally choose \(n_1\) sufficiently large in terms of all
preceding constants and above the thresholds of the cited results.
In particular, all subsequent \(O(1)\) and \(O(\sqrt n)\) errors
are smaller than \(\eta n\).
Apply Lemma~\ref{lem:cleaning}, and let \(R,P,A,B\) be the resulting sets. Set
\(Q=R\setminus P\) and \(J=G[A\cup P]\). Since \(|X|,|Y|=n/2+O(\sqrt n)\) and \(|R|\le s\),
\(|A|,|B|\ge n/3\). Moreover,
\(\overline e(A,B)\le \gamma n\), and
\(\Delta(G[A])+\Delta(G[B])\le\eta n\). For \(x\in A\), we have
\(\overline d_B(x)\le d_A(x)+|R|+\gamma\) by \eqref{eq:replacement}; for every \(y\in B\),
\(\overline d_A(y)\le d_B(y)+|R|+\gamma\). Hence
\[
 \max_{x\in A}\overline d_B(x)+
 \max_{y\in B}\overline d_A(y)
 \le\eta n+2|R|+2\gamma\le2\eta n<\varepsilon n.
\]
Thus \(A,B\) satisfy Lemma~\ref{lem:pinnededge}\textup{(a)--(c)} and
\(\Delta(G[A])\le\varepsilon n\) in \textup{(d)}. We next verify the edge
bound in \textup{(d)}.

Since \(z\) is a centroid, \(\sigma\le\lfloor k/2\rfloor\). By K\H{o}nig's theorem,
\(\tau(T)=\nu(T)\le\lfloor k/2\rfloor\), and hence
\(k-\tau(T)\ge\sigma\). Since
\(|R|\le s\) and \(|X|=n/2+O(\sqrt n)\), we have
\(|A|=n/2+O(\sqrt n)\). By Lemma~\ref{lem:cleaning}\textup{(i)},
\[
 e(G[A])-\bigl(\sigma-\tfrac12\bigr)|A|
 \ge \frac{k-\tau(T)-\sigma}{2}n
      +\bigl(\tfrac14-\eta\bigr)n-O(\sqrt n)>0.
\]
Thus
\(e(G[A])\ge(\sigma-1+1/2)|A|\), as required in
Lemma~\ref{lem:pinnededge}.

By Lemma~\ref{lem:decomp},
\(H=J\vee\ind{|B|+|Q|}\), with independent vertex set \(B\cup Q\),
does not contain \(F\) as a subgraph. Since \(v(J)+|B|+|Q|=n\),
\(e(H)=v(J)(|B|+|Q|)+e(J)\le\Psi(n)\). Counting edges using
\(V(G)=A\dot\cup P\dot\cup B\dot\cup Q\), we obtain
\[
 \begin{aligned}
 e(G)-e(H)
 &=e(G[B])+e(G[Q])-\overline e(A,B)
   -\overline e(Q,A\cup P)+e(B,R)-|B||P|\\
 &\le-\overline e(Q,A\cup P)+3\eta n,
 \end{aligned}
\]
by Lemma~\ref{lem:cleaning}\textup{(i)--(ii)},
\(|Q|\le s\), and the choice of \(n_1\). Since
\(e(G)\ge\Psi(n)\ge e(H)\), it follows that
\begin{equation}\label{eq:Q-missing}
 \overline e(Q,A\cup P)\le3\eta n.
\end{equation}

We next prove that \(P=R\). Otherwise, by
Lemma~\ref{lem:cleaning}\textup{(iii)}, there is an edge \(uv\) with
\(u\in Q\), \(v\in B\), and
\(d_Y(v)\Delta(G[A])\le\eta n\). By \eqref{eq:Q-missing},
\(\overline d_A(u)\le3\eta n<\varepsilon n\). Since \(A\subseteq X\), we obtain from \eqref{eq:replacement} that
\[
 \overline d_A(v)\Delta(G[A])
 \le(d_Y(v)+\gamma)\Delta(G[A])
 \le(\gamma+1)\eta n<\varepsilon n.
\]
This contradicts Lemma~\ref{lem:pinnededge}. Therefore \(P=R\),
\(|R\cap Y|\Delta(G[A])\le\eta n\), and \(J=G[A\cup R]\) does not contain any member of \(\M(F)\) as a subgraph.

Suppose that \(B\) contains an edge. The induced subgraph on
\(\{b\in B:d_B(b)>0\}\) has average degree at most \(2(k-1)\), by
Lemma~\ref{lem:cleaning}\textup{(i)} and Lemma~\ref{lem:decomp}.
Choose \(v\in B\) with \(1\le d_B(v)\le2(k-1)\). Since
\(Y=B\dot\cup(R\cap Y)\),
\(d_Y(v)\le2(k-1)+|R\cap Y|\). By \eqref{eq:replacement},
\(\overline d_A(v)\Delta(G[A])
\le(\gamma+2k-1)\eta n<\varepsilon n\).
For every \(u\in N(v)\cap B\), we have
\(\overline d_A(u)\le\Delta(G[B])+|R|+\gamma\le2\eta n<\varepsilon n\)
by Lemma~\ref{lem:cleaning}\textup{(i)}, \eqref{eq:replacement}, and
the choice of \(n_1\). The edge \(uv\) contradicts
Lemma~\ref{lem:pinnededge}. Hence \(B\) is
independent, and
\[
 \Psi(n)\le e(G)\le v(J)|B|+e(J)
 \le v(J)|B|+g(v(J))\le\Psi(n).
\]
Equality gives \(\overline e(V(J),B)=0\), \(e(J)=g(v(J))\), and
\(v(J)|B|+g(v(J))=\Psi(n)\). Since
\(V(G)=V(J)\dot\cup B\), we obtain
\(G=J\vee\ind{|B|}\in\mathcal G(n)\), a contradiction.
\end{proof}

\begin{proof}[Proof of Theorem~\ref{thm:main}]
By Lemma~\ref{lem:decomp}, every member of \(\mathcal G(n)\) is
\(F\)-free and has \(\Psi(n)\) edges. In particular,
\(\ex(n,F)\ge\Psi(n)\). We complete the proof by
Simonovits's progressive induction~\cite[p.~284]{Simonovits1968},
using the following direct descent.

Fix \(t_0\) as in Lemma~\ref{lem:core}, put \(h=2t_0\), and let
\(n_1\) be given by Lemma~\ref{lem:reduction}. Define
\[
 \phi(n)=\ex(n,F)-\Psi(n)\ge0,
 \qquad K=\max\{\phi(r):1\le r\le n_0\},
\]
where \(n_0>\max\{n_1,h\}\) is a fixed integer.
If \(n>n_0\) and some \(G\in\EX(n,F)\) does not belong to
\(\mathcal G(n)\), Lemma~\ref{lem:reduction} gives
\[
 \ex(n,F)\le\ex(n-h,F)+t_0(n-t_0)
              +(k-1)t_0-\binom{k}{2}-1.
\]
Together with the increment bound in Lemma~\ref{lem:decomp}, this
implies
\begin{equation}\label{eq:progressive-descent}
 \phi(n)\le\phi(n-h)-1.
\end{equation}

We claim that every \(G\in\EX(n,F)\) belongs to \(\mathcal G(n)\)
whenever \(n>n_0+hK\). Otherwise, start with such a graph outside
\(\mathcal G(n)\) and apply \eqref{eq:progressive-descent}.
Since \(\phi(n-h)>0\), no extremal graph of order \(n-h\) can
belong to \(\mathcal G(n-h)\); hence the same inequality can be
applied again as long as the current order exceeds \(n_0\).
Let \(q\) be the least positive integer for which
\(r=n-qh\le n_0\). Then \(n_0-h<r\le n_0\) and \(q>K\), so
iteration yields
\[
 0\le\phi(n)\le\phi(n-h)-1\le\cdots
 \le\phi(n-qh)-q\le K-q<0,
\]
a contradiction. Thus \(\EX(n,F)\subseteq\mathcal G(n)\) for
all sufficiently large \(n\), which also gives
\(\ex(n,F)=\Psi(n)\). The reverse inclusion follows from the
first sentence of the proof.
\end{proof}

\section*{Declaration of generative AI and AI-assisted technologies in the paper preparation process}

The authors first obtained the results for the Cartesian products of an odd cycle with a path and with a star. Using similar techniques and with the assistance of GPT-6, they extended these results to arbitrary trees. GPT-6 also assisted with the writing and typesetting of this paper. The authors reviewed the AI-assisted material and take full responsibility for the correctness of the results and proofs and for the content of this paper.

\end{document}